\documentclass[12pt]{amsart}

\usepackage{graphicx}
\usepackage[bottom,flushmargin,hang,multiple]{footmisc}

\font\bbbld=msbm10 scaled\magstephalf

\newcommand{\bfR}{\hbox{\bbbld R}}
\newcommand{\bfS}{\hbox{\bbbld S}}
\newcommand{\e}{\varepsilon}
\newcommand{\goto}{\rightarrow}
\newcommand{\ol}{\overline}
\newcommand{\ul}{\underline}
\newcommand{\be}{\begin{equation}}
\newcommand{\ee}{\end{equation}}
\newcommand{\bea}{\begin{eqnarray}}
\newcommand{\eea}{\end{eqnarray}}
\newcommand{\diam}{\operatorname{diam}}

\newtheorem{theorem}{Theorem}[section]
\newtheorem{lemma}[theorem]{Lemma}
\newtheorem{proposition}[theorem]{Proposition}

\theoremstyle{definition}
\newtheorem{definition}[theorem]{Definition}
\newtheorem{example}[theorem]{Example}

\theoremstyle{remark}
\newtheorem{remark}[theorem]{Remark}

\theoremstyle{claim}
\newtheorem{claim}[theorem]{Claim}

\numberwithin{equation}{section}

\begin{document}
\setlength{\baselineskip}{1.2\baselineskip}

\title[A personal tribute to Louis Nirenberg ]
{ A personal tribute to Louis Nirenberg:\\
 February 28, 1925--January 26, 2020. 
} 

\author{Joel Spruck  }
\address{Department of Mathematics, Johns Hopkins University,
 Baltimore, MD 21218}
\email{js@math.jhu.edu}

\thanks{This paper is an expanded version of  a talk of the same title given at the Geometry Festival Stony Brook, April 23, 2021.}

\begin{abstract}
I  first met Louis Nirenberg  in person in 1972 when I became a Courant Instructor.
He was already a celebrated mathematician and a suave sophisticated New Yorker, even though he was born in Hamilton, Canada and grew up in Montreal. In this informal style paper I will describe some of his famous papers, some of our joint work and other work he inspired. I will concentrate on some of Louis' work inspired by geometric problems  beginning around 1974, especially the method of moving planes and implicit fully nonlinear elliptic equations. I  have also sprinkled throughout some comments on his character and personality that I believe contributed to his great success.
\end{abstract}

\maketitle

\vspace{-.4in}

\section{Introduction: California meets New York 1972. }

\label{sec1}


I  first met Louis Nirenberg in person \footnote{Louis was for a few years an editor of JDG and I sent him the paper  \cite{ball} in 1971.} in 1972  when I became a Courant Instructor after my Ph.d work at Stanford.
He was already a celebrated mathematician and a suave sophisticated New Yorker, even though he was born in Hamilton, Canada and grew up in Montreal. I was born in Brooklyn, New York and a post-doc at Courant was a return home for me. I had been greatly influenced by my time at Stanford in the late sixties by the hippie culture and the political anti-war activism. The picture below shows more or less how Louis and I comparatively  looked when I arrived at Courant institute. \\

\begin{center}
\includegraphics[width=3cm,  height=3cm]{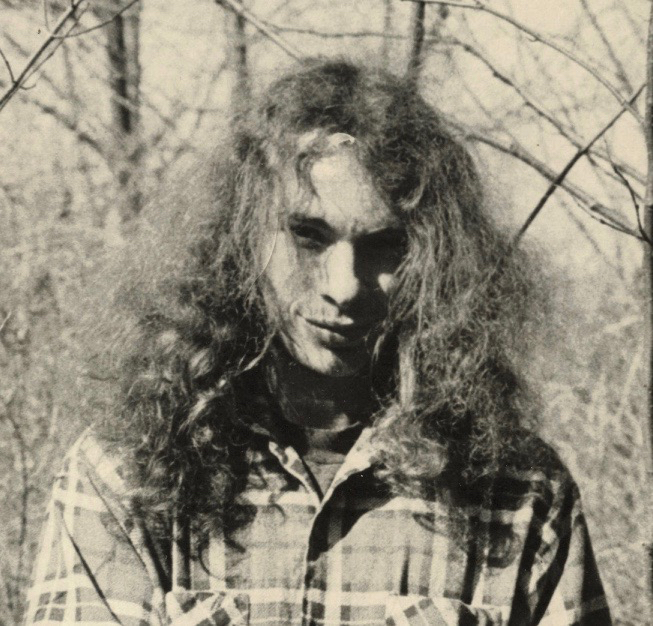}
\includegraphics[width=2.5cm,  height=3cm]{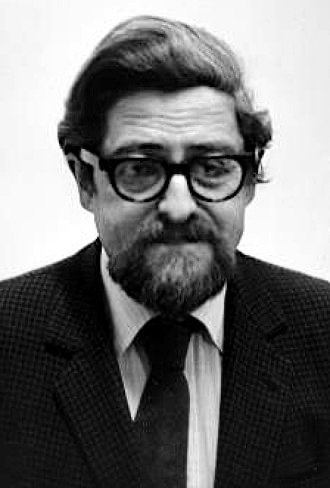}
\end{center}

 Louis accepted and welcomed me without hesitation as he did with almost everyone he met. One of his greatest and most endearing attributes which served him well throughout his entire life was his openness to young people and new ideas. He lived on the upper west side in a large New York style apartment and loved books and movies and of course food and travel. He was funny with a wry sense of humor which was often self-deprecating. If he really liked an idea, he would say, ``that's good enough to steal''. He once took me to Veniero's, a famous Italian pastry and coffee place in the East village and as we were stepping up to the counter to order he told me he wasn't going to treat me because he was cheap! In fact he was an incredibly generous person who gave freely of his time and who often invited the young faculty at Courant to dinner at his home. \footnote{At a Courant Christmas party I was surprised to see that he also loved to dance!}

 In this informal style paper, I will describe some of Louis' famous papers, some of our joint work and other work he inspired. \footnote{I  have also sprinkled throughout some comments on his character and personality that I believe contributed to his success.}
For reasons of exposition, I will not follow a strict
chronological order and 
I will concentrate on some of Louis' work inspired by geometric problems  beginning around 1974, especially the method of moving planes and implicit fully nonlinear elliptic equations.
During the twenty year period 1953-1973 \footnote{Louis received his Ph.d in 1949 and his papers started appearing in rapid fire in 1953.} he produced an incredible body of work in many fields of pde, geometric analysis and complex geometry starting with his famous thesis work \cite{N.thesis} solving the classical Weyl and Minkowski problem via a continuity method, the Newlander-Nirenberg theorem on the integrabilty of almost complex structures \cite{NewN} and its application to the deformation  of complex structures \cite{KoNS}, his work with Bers on the representation of solution of linear elliptic equation in the plane \cite{BersN1, BersN2} (see also \cite{N1953}), his work with Agmon and Douglis \cite{DoN, ADN1, ADN2} on linear elliptic boundary value problems, the John-Nirenberg inequality and BMO \cite{JoN}, his work with Kohn on non-coercive boundary value problems and pseudo- differential operators  \cite{KoN1, KoN2, KoN3} , his work with Tr\`{e}ves on local solvability \cite{TN, TN1, TN2}, his many papers on Sobolev type inequalities (see for example \cite{N1959}), his work with Caffarelli and R. Kohn on partial regularity for suitable weak solutions of the Navier-Stokes equation \cite{CKoN} and the list goes on.  Of course I cannot begin to talk about this work but 
fortunately there are a number of excellent appreciations of his work, for example \cite{1996.medal} ,  \cite{YYLi}, \cite{Don},  \cite{Riviere},    \cite{AMS},  \cite{Kohn.Abel},  \cite{Vasquez}.
Louis loved to collaborate and I apologize for omitting many other important results. \footnote{Louis' most frequent collaborators according to Mathscinet were Luis Caffarelli, Henri Berestycki, Yanyan Li, Haim Brezis, Joel Spruck, David Kinderlehrer, Joseph J. Kohn, Francois Tr\`{e}ves, Shmuel Agmon, Peter Lax,  Italo Capuzzo-Dolcetta, Avron Douglis, Lipman Bers, Ivar Ekeland, Basilis Gidas, Robert V. Kohn and Wei-Ming Ni,  Donald C. Spencer, S.S. Chern, Fritz John, Kunihiko Kodaira, Harold I. Levine, Charles Loewner, Guido Stampacchia, Srinivasa R. S. Varadhan, Homer F. Walker, Sidney M. Webster and Paul Yang.} 
 His brilliant student and frequent collaborator Yanyan Li was especially devoted to Louis in his later years.

\section{Louis' work on elliptic systems finds some novel applications.} 
Louis   with Agmon and Douglis  introduced an incredibly general notion of a linear elliptic system 
$$\sum_{j=1}^N L_{kj}(y,1/i  \ \partial)u^j(y)=f_k(y),\, 1\leq k\leq n$$
that is somewhat mysterious and hard to understand for the non-expert.  Here the unknowns are $u^1,\ldots u^n$, defined in a bounded domain $\Omega \subset \bfR^N$ or sometimes $\bfR_+^N$.  It involves assigning  integer weights  $s_k\leq 0$ to each equation and $t_k \geq 0$ to each unknown with 
$$\max_k s_k=0,\,\text{order}\,L_{kj}\leq s_k+t_j\,\,\text{(consistency)},$$
 in order to determine the principal part $L'_{kj}$ of the system, that is the part of $L_{kj}$ where the order  is precisely $s_k+t_j$. Ellipticity  then becomes a maximal rank condition on $L'_{kj}(y,\xi)$. One of the oddities of this notion of ellipticity (as was  pointed out in the paper of Douglis and Nirenberg \cite{DoN}) is that it can be destroyed by a non-singular transformation of the equations and the dependent variables and is thus  seemingly dependent on the particular way the equations and unknowns are presented. We will in fact exploit this in the Lewy example below!
 
I want to mention two novel applications of this work to what is  called `` free boundary regularity'' taken from a joint paper \cite{KNS1} of Louis with David Kinderlehrer and me.

\subsection{An example of Hans Lewy}
The first example is due to Hans Lewy. Louis was a great admirer of Hans Lewy and was well acquainted with his work.
\vspace{.1in}

\begin{center}
\includegraphics[width=3cm,  height=2.5cm]{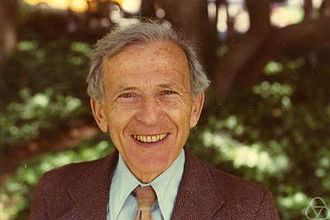}
\end{center}

Let $u,v \in C^2(\Omega\cup S)$ be solutions of the system
\begin{eqnarray}
\label{eq2.10}&\Delta u=0, \,\, \,\Delta v +\lambda(x) v=0 \,\,\text{in $\Omega$},\\
\label{eq2.20}&u=v,\,\, u_{x_n}=v_{x_n}\,\,\text{ on $S$}.
\end{eqnarray}
where $\lambda(x)\neq 0$ is analytic and $\partial \Omega$ is partly contained in the hyperplane $S=\{x_n=0\}$. Lewy proved that for $n=2$, $u$ and $v$ extend analytically across $S$. This is very surprising since the boundary conditions, which say that $u$ and $v$ share the same boundary conditions are ``not coercive''. To show that coercivity is present, we introduce the ``splitting'',
$w=u-v$ and rewrite the system \eqref{eq2.10}, \eqref{eq2.20} in terms of $u$ and $w$:
\begin{eqnarray}
\label{eq2.30}\Delta u=0,\\
\label{eq2.40}\Delta w +\lambda(x)(w+u)=0,\\
\label{eq2.50}w=w_{x_n}=0\,\,\text{on on $S$}.
\end{eqnarray}

Using the ADN general theory, we assign the weights $s=0$ to  equation \eqref{eq2.30} and s=\ -2 to equation \eqref{eq2.40} and the 
weights $t_u=2$ and $t_w=4$ to the unknowns. These weights are ``consistent'' and with this choice, the principal part of the system is 

\begin{eqnarray}
\label{eq2.60} \Delta u=0,\\
\label{eq2.70} \Delta w +\lambda(x)u=0,
\end{eqnarray}
which is certainly elliptic.

The crucial point to note is that now the boundary conditions \eqref{eq2.50} are coercive for the system.
Indeed we can eliminate $u$ entirely from \eqref{eq2.60}, \eqref{eq2.70} since $u=-(1/\lambda)\Delta w$ and obtain
$$\Delta^2 w+2\lambda \nabla(1/\lambda)\cdot \nabla \Delta w +\Delta(1/\lambda)\Delta w=0.$$
The boundary conditions $w=w_{x_n}=0$ are now just Dirichlet conditions for this fourth order equation which are well known to be coercive. The end result is that $u$ and $v$ are analytic in $\Omega\cup S$  (or $C^{\infty}$ if $\lambda \in C^{\infty}$). One point to emphasize is that {\em it is not necessary to eliminate $u$ } and often not possible in other related  problems.

\subsection{The regularity of the liquid edge}
Our second  example concerns the so-called ``liquid edge'' of a stable configuration of soap films.
Consider a configuration of three minimal surfaces in $\bfR^3$ meeting along a $C^{1+\alpha}$ curve, the liquid edge $\gamma$, at equal angles of $2\pi/3$.
Such a configuration represents one of the stable singularities that soap films can form (Jean Taylor \cite{Taylor}, J.C.C Nitsche \cite{Nitsche}).

\begin{center}
\includegraphics[width=4cm, height =4cm]{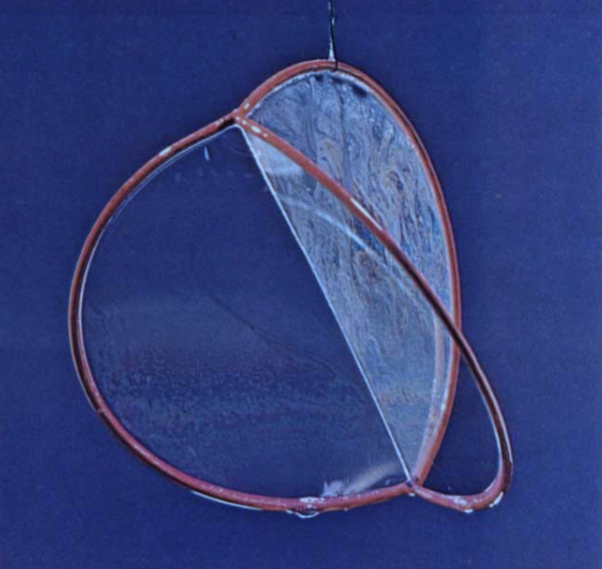}
\end{center}

We can reformulate this configuration as a classical free boundary problem.
With the origin of a system of coordinates on the curve $\gamma$, let us represent the three surfaces as graphs over the tangent plane to one of them at the origin.
Denote by $\Gamma$ the orthogonal projection of $\gamma$ on this plane. Two of the functions, say $u^1,\,u^2$ will be defined on one side $\Omega^+$ of $\Gamma$ while the third $u^3$ will be defined on the opposite side $\Omega^-$ of $\Gamma$.

We now extrapolate and suppose $\Omega^{\pm} \subset \bfR^n$ and $\Gamma$ is a $C^{1+\alpha}$ hypersurface.
Suppose that the graphs of $u^1$ and $u^2$ meet at angles $\mu_1, \,\mu_2$.
Then we have an overdetermined system for the $u^i,\,i=1,2,3$:

\begin{eqnarray}
Mu^1=Mu^2&=&0 \,\,\,\,\text{in $\Omega^+$},\\
Mu^3&=&0 \,\,\,\,\text{in $\Omega^-$},\\
u^1=u^2&=&u^3\,\,\,\text{on $\Gamma$},\\
\frac{\nabla u^j \cdot\nabla u^3+1}{\sqrt{1+|\nabla u^j|^2} \sqrt{1+|\nabla u^3|^2}}&=&\cos \mu_j \,\,\,\,j=1,2,\\
\nabla u^3(0)&=&0,
\end{eqnarray}
where
$$Mu=(\delta_{ij}-\frac{u_{x_i}u_{x_j}}{1+|\nabla u|^2})u_{x_i x_j} $$
is the minimal surface operator in nondivergence form.

\begin{theorem}If three distinct minimal hypersurfaces in $\bfR^{n+1}$ meet along an $(n-1)$ dimensional $C^{1+\alpha}$ hypersurface $\gamma$ at constant angles, then $\gamma$ is  analytic.
\end{theorem}

One has to introduce $w(x)=(u^2-u^1)(x)$ and the transformation $y=(x', w(x)),\,x\in \Omega^+$, the 
so-called zeroth order partial Legendre transform \cite{KN}.
The mapping $x\rightarrow y$ transforms a neighborhood of $0$ in $\Omega^+$ into a neighborhood $U\subset \{y:y_n>0\}$ and a portion of $\Gamma$ into a  $S\subset \{y:y_n=0\}$. It has an inverse
$$ x=(y', \psi(y)),\,\,y\in U\cup S.$$

We associate to this inverse a reflection mapping\footnote{The simplicity of the reflection mapping in combination with the zeroth order partial Legendre transform and the ADN theory is the real point of this example.}
$$x=(y', \psi(y)-Cy_n),\,\,y\in U\cup S,$$
for $C>\sup_U |\nabla \psi(y)|$. Now define
$$\phi^+(y)=u^1(x), \,x\in \Omega^+,\,\, \phi^-(y)=u^3(x),\,\,x\in \Omega^-.$$
and note that $u^2(x)=(w+\phi^+)(y)=y_n+\phi^+(y)$.
Then $\phi^+(y),\,\phi^-(y),\,\psi(y)$ satisfy  a coercive elliptic system in the ADN sense (see \cite{KNS1} for the technical details) which implies the analyticity.

\section{The influence of Alexandrov  and Serrin on Nirenberg's work.}
A central question in differential geometry has been to classify all possible ``soap bubbles'', or more formally put:
classify all closed constant  mean curvature hypersurfaces in $\bfR^{n+1}$ (or more generally in special Riemannian manifolds $M^{n+1}$). Amazingly, the first such result was proven by J. H. Jellett \footnote{John Hewitt Jellett 1817-1888 was an Irish mathematician who in 1849 \cite{Jellett1} studied how fixing a non-asymptotic curve in an analytic surface would render it infinitesimally rigid.} \cite{Jellett2} in 1853.

\vspace{.1in}

\begin{center}
\includegraphics[ width=3cm, height=4cm]{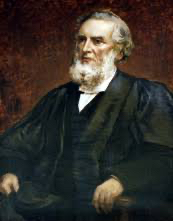}
\end{center}

A starshaped closed surface $M$ in $R^3$ with constant positive mean curvature is the standard round sphere.
A modern presentation of his proof, which extends to  $\bfR^{n+1}$ goes as follows. Let  $X:M^n \rightarrow \bfR^{n+1}$ be the position vector and $N$ the outward normal to $M$. We assume $X(M)$ has constant mean curvature $H>0$ and is starshaped about the origin, i.e. $X\cdot N \geq 0$. If $\Delta_M$ is the Laplace-Beltrami operator and $A$ the second fundamental form of $M$, then

\vspace{-.2in}

\begin{eqnarray}
\label{J1} \Delta_M X&=&-nH N,\\
\label{J2} \Delta_M N&=&-|A|^2 N.
\end{eqnarray}
and \eqref{J1}, \eqref{J2} imply
\begin{eqnarray}
\label{J3} \Delta_M |X|^2/2&=&n-nH N,\\
\label{J4} \Delta_M X\cdot N&=&nH-|A|^2 X\cdot N.
\end{eqnarray}

Then \eqref {J3}, \eqref{J4} yields
\be \label{J5} \Delta_M (H|X|^2/2-X\cdot N)=(|A|^2-nH^2)X\cdot N\geq 0. \ee
Integrating \eqref{J5} over M we conclude $|A|^2=nH^2$ or $M$ is totally umbilic, so a sphere.

\subsection{Heinz Hopf  lectures at Courant and Stanford.}


\begin{center}
\includegraphics [width=3cm, height =4cm]{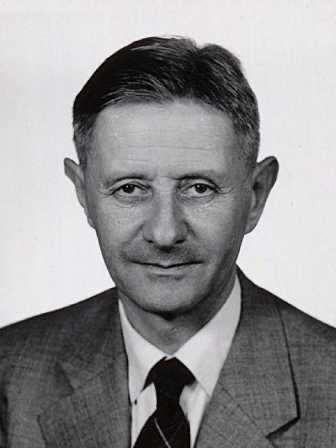}
\includegraphics[ width=3.5cm, height =3.5cm]{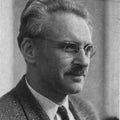}
\end{center}

In his visit to Courant and Stanford in 1955-1956, Hopf  lectures on various topics in differential geometry in the large and presents his ingenious proof that if $M$ is a closed immersed constant mean surface in $R^3$ with genus 0, then $M$ is a round sphere. He also sketched the proof of a new result  (not published at the time) of Alexandrov:

\begin{theorem}A closed embedded hypersurface $M^n$ in $R^{n+1}$ of constant mean 
curvature is a round sphere.\footnote{In 1982, Wu-Yi Hsiang \cite{Hsiang} produced infinitely many closed immersed cmc examples in $\bfR^n,\, n\geq 4$. In 1984 Henry Wente \cite{Wente} produced the first genus one example in $\bfR^3$.}
\end{theorem}
\vspace{.1in}

Hopf goes on to speculate, `` it is my opinion that this proof by Alexandrov, especially the geometric part, opens up important new aspects in differential geometry in the large".

Alexandrov's idea is to show using the maximum principle that an embedded hypersurface has a plane of symmetry for any direction $e_1$ and thus is a sphere. Note that any open set of directions suffices.

\subsection{An important paper of Jim Serrin.}
Jim Serrin (1926-2012) was a  brilliant innovator in elliptic theory, the calculus of variations, fluids and mechanics.\\

\begin{center}
\includegraphics[width=4cm, height =4cm]{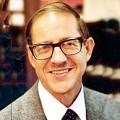}
\end{center}

In 1971, Serrin published a paper in which he showed that  if $u$ is a solution of the overdetermined (free boundary) problem
$$\Delta u = -1\,\,\text{in $\Omega,\, u = 0,\, u_{\nu}=c<0 $ on $\partial \Omega$},$$ (with $\Omega$ a $C^2$ domain)
then $\Omega$ is a ball.  A more physically appealing example in Jim's paper is that if a liquid in a capillary tube has both constant height and constant angle on the boundary, then the tube must have circular cross section.

The important part of Jim's paper was that he used Alexandrov's method of moving planes  and developed an important improvement of Hopf's boundary point lemma as well as a continuous (sweeping) use of the maximum principle. 
These techniques were simplified and improved by Gidas, Ni and Nirenberg \cite{GNN1} in their first paper but Serrin laid out the technical groundwork.\footnote{It is very surprising to me that Jim missed the greater application of his idea to proving symmetry of solutions of elliptic pde and I can only speculate that he was intensively working on other things and moreover such questions ``were not in the air'' at the time.}

\subsection{Symmetry for solutions of elliptic pde; motivating questions.}
Many problems in both Yang-Mills  theory, astrophysics and reaction-diffusion equations are modeled by equations of the form
$\Delta u +f(u)=0$. For example the power nonlinearity $f(u)=u^{\alpha}$ occurs frequently. The range $1\leq \alpha< \frac{n+2}{n-2} $ is called the subcritical range because of the Sobolev embedding theorem while the range $\alpha>\frac{n+2}{n-2}$ is called supercritical. The case $\alpha=\frac{n+2}{n-2}$ is particularly important because the equation becomes conformally invariant. This is the conformally flat case of the famous Yamabe problem.

Suppose we consider the Dirichlet problem
$$\Delta u +f(u)=0, \, u> 0 \,\,\text{in $B=B_1(0),\,u=0$ on $\partial B$}.$$
Is $u$ radially symmetric, i.e $u=u(|x|)$?
There is a simple counterexample if we allow $u$ to change sign. Take $f(u)=\lambda_k u$ where $\lambda_k$ is the kth Dirichlet eigenvalue of B. Then the eigenfunctions are not radially symmetric.

Let $\Delta u+f(u)=0,\, u\geq 0$ in $R^n$. Is $u$ radially symmetric? 
The famous equation $\Delta u+u^{\frac{n+2}{n-2}}=0,\, u>0$  in $R^n$ has the explicit solutions
$$u(x)=(\frac{\lambda\sqrt{n(n-2)}}{\lambda^2+|x-x_0|^2})^{\frac{n-2}2},\,\, \lambda>0.$$
On the other hand the subcritical equation $\Delta u+u^{\alpha}=0,\, u> 0$ in $R^n$ has only the trivial solution (Gidas-Spruck \cite{GS}).

\subsection{Singular solutions.}

There is also a complicated family of singular solutions of the form  $u(x)=r^{-\frac{n-2}2}\psi(t)$ where $r=|x|,\,t=-\log r$ and $\psi(t)$ is a periodic translation invariant solution of the ODE:
$$\psi''-(\frac{n-2}2)^2 \psi+\psi^{\frac{n+2}{n-2}}=0.$$
The simplest singular solution is
$$u(x)=\frac{k}{|x|^{\frac{n-2}2}},\, k=({\frac{n-2}2})^{{\frac{n-2}2}}.$$
Thus the classification of singular solutions is quite challenging.

\subsection{A general symmetry result of Nirenberg and Berestycki for  a bounded domain.}

Nirenberg refined his original method with Gidas and Ni in a paper with Berestycki  \cite{BN1} by incorporating the Alexandrov-Bakelman-Pucci (ABP) maximum principle into the argument. We sketch the simplest case of this improved proof which again uses another important idea of Alexandrov. These papers have had tremendous impact.

\begin{theorem} Let $\Omega$ be a bounded domain which is convex in the $e_1$ direction and symmetric with respect to the plane $x_1=0$. Suppose $u\in C^2(\Omega)\cap C^0(\ol{\Omega})$
 satisfies \footnote{It suffices that $u\in W^{2,n}_{\text{loc}}(\Omega)$} 
$$\Delta u+f(u)=0,\, u>0 \,\,\text{in $\Omega,\,u=0$ on $\partial \Omega$},$$
where $f(t)$ is Lipschitz. Then $u(x_1,x')=u(-x_1, x')$ and $u_{x_1}(x)<0$ for any $x$ with $x_1>0$.
\end{theorem}

\subsection{Alexandrov's generalized gradient map and Monge-Ampere measure}

Alexandrov introduced the normal mapping (subdifferential) of  $u\in C^0(\Omega)$
$$\partial u(x_0)=\{p\in R^n: u(x)\geq u(x_0)+p\cdot (x-x_0)\},\,
\partial u(E)=\cup_{x\in E} \partial u(x)\,\,\text{for $E\subset \Omega$}.$$
Of course, $\partial u(x_0)$ may be empty.
The lower contact set $\Gamma_u:=\{x\in \Omega: \partial u(x)\neq \emptyset \}$.\\
If $u\in C^1$ and $x\in \Gamma_u$, then $\partial u(x)=\nabla u(x)$.
If $u\in C^2$ and $x\in \Gamma_u$, then $\nabla^2 u(x) \geq 0$.

\begin{example} $\Omega=B_r(0),\, u(x)=\frac{h}R(|x|-R)$. Then
$$\partial u(x)=\begin{cases}\frac{h}R \frac{x}{|x|} & x\neq 0\\
\vspace{-.1in}{}&\\
 \ol{B_{\frac{h}R(0)}} & x=0 \end{cases}$$
 \end{example}

\begin{theorem} If $u\in C^0(\Omega)$, then $\mathcal{S}=\{E\subset \Omega: \partial u(E) \text {is Lebesgue measurable}\}$ is a Borel $\sigma$ algebra. The set function $\mathcal{M}u(E)=|\partial u(E)||$ is called the Monge-Ampere measure associated to $u$ and if $u\in C^2(\Omega)$ is convex,
$$\mathcal{M}u(E)=\int_E \text{det}(D^2 u) .$$
\end{theorem}
See the books of Figalli \cite{Figalli} and Gutierrez \cite{Gutierrez} for a detailed discussion of the ideas of Alexandrov.

The following beautiful argument of Cabr\'{e} \cite{Cabre} \footnote{ Xavier Cabr\'{e} received his Ph.d with Louis in 1994.} utilizes Alexandrov's contact set to prove the classical isoperimetric inequality. Given $\Omega$, solve the Neumann problem
\begin{eqnarray*}
&\Delta u=1 ~ \mbox{in} ~\Omega \\
&u_{\nu}=c=\frac{|\Omega|}{|\partial \Omega|}~\mbox{on} ~\partial \Omega
\end{eqnarray*}
\begin{claim} The normal map $\partial u(\Gamma_u)$ contains a ball of radius c.
\end{claim}
For if a hyperplane $\text{graph}(p\cdot x+k)$ is translated upwards in $\bfR^{n+1}$ ``from $-\infty$'' has first contact  on $\partial \Omega$, then $|p| \geq |\nabla u| \geq c,$ proving the Claim.
Hence,
\[\omega_n c^n \leq \int_{\Gamma_u} \det{D^2 u} \leq \int_{\Gamma_u}(\frac
{\Delta u}n)^n 
\leq \frac1{n^n}|\Omega|~, \,\,\text{or}\]
$$ |\partial \Omega| \geq n\omega_n^{\frac1n}|\Omega|^{\frac{n-1}n}~.$$

\subsection{A simple  version of the ABP maximum principle.}
\begin{lemma}
Assume $\Delta u+c(x) u \geq f $ in $\Omega,\, c(x)\leq 0, \, u\leq 0$ on $\partial \Omega$.
Then $M:=\sup_{\Omega}u\leq Cd\|f^-\|_{L^n(\Omega)},\, d=\text{diam}(\Omega)$.
\end{lemma}
\begin{proof} Assume $M>0$ is achieved at an interior point $x_1\in \Omega$. Let $v=-u^+$. Then $v<0$ on $\Gamma_v$, the lower contact set  of $v$ and $-M=v(x_1),\, v=0$ on $\partial \Omega$. Hence
$$\Delta v(x)=-\Delta u(x)\leq -\|c^+\|_{L^{\infty}(\Omega)} f(x)+c(x)u(x)\leq f^-(x),\, x\in \Gamma_v.$$
\begin{claim} $B(0, M/d)\subset \nabla v(\Gamma_v)$.
\end{claim}
Let $|p|<M/d$ and translate the plane $\text{graph}(p\cdot x+k)$ ``up from $-\infty$'' until there is a first contact with the graph of $v$ at a point
$x_0 \in \ol{\Omega}$. Then $x_0 \not \in \partial \Omega$ for otherwise (since this would imply $v(x_0)=0$)
$$-M=v(x_1)\geq p\cdot (x_1-x_0)\geq -|p| d>-M,\,\,\text{a contradiction.}$$
Therefore $p=\nabla v(x_1)$ proving the Claim. 

Hence 
$$\omega_n (M/d)^n\leq \int_{\Gamma_v}\text{det}(D^2 v)\leq \int_{\Gamma_v}(\Delta v/n)^n \leq \|f^-\|_{L^n(\Omega)}^n,$$
and the lemma follows. \end{proof}

\subsection{The maximum principle for domains of small volume.} 
We now prove the simplest case \footnote{See \cite{GT} for a proof of the most general version} of the Alexandrov-Bakelman-Pucci maximum principle for $\Delta u+c(x)u\geq 0$ and drop the assumption $c(x)\leq  0$. To accomplish this we write $c=c^+-c^{-},\, \Delta u -c^- u \geq -c^- u^+$. Applying the Lemma gives
$$M\leq Cd\|c^+ u^+\|_{L^n(\Omega)}\leq Cd\|c^+\|_{L^{\infty}(\Omega)} |\Omega|^{\frac1n}M\leq M/2$$
for $|\Omega|\leq (2Cd \|c^+\|_{L^{\infty}(\Omega)} )^{-n}$ and thus $M=0$. We have proved with no assumption on the sign of $c(x)$:
\begin{theorem} \label{ABP} Assume $\Delta u +c(x)u \geq 0$ in $\Omega,\, u\leq 0 $ on  $\partial \Omega$. Then there is a positive constant $\delta$ depending on $n,\, \diam(\Omega), \,\|c^+\|_{L^{\infty}(\Omega)} $ so that if $|\Omega|\leq \delta $, then $u\leq 0$ in $\Omega$.
\end{theorem}

\subsection{The proof of symmetry by moving planes in the $x_1$ direction. }
We will prove that 
$$u(x_1, x')<u(x_1^*, x')\,\,\text{ for all $x_1>0$ and $-x_1<x_1^*<x_1$.}$$
This implies monotonicity in $x_1$ for $x_1>0$. Letting $x_1^*\goto -x_1$ implies
$$ u(x_1,y)\leq u(-x_1,y) \,\, \text{for any $x_1>0$.}$$
Replacing $e_1$ by $-e_1$ gives reflectional symmetry
$$u(x_1,y)=u(-x_1,y)\,\,\text{for any}\,\,(x_1,y)\in \Omega.$$
\noindent Step 1. Moving the plane $T^{\lambda}:=\{ x_1=\lambda\}$ from ``infinity''. \\
Let $a=\sup_{\Omega} x_1$ and for $\lambda<a$, set $\Sigma^{\lambda}:=\{x\in \Omega: x_1>\lambda\}$.\\
Given $x=(x_1,y)\in \Omega$, let $x^{\lambda}=(2\lambda-x_1,y)$ be its reflection in $T^{\lambda}$.\\
For $x\in \Sigma^{\lambda}$, set $u^{\lambda}(x)=u(x^{\lambda})$ and $w=w^{\lambda}(x)=u(x)-u^{\lambda}(x)$.\\
By the mean value theorem,
$$\Delta w=-\frac{f(u)-f(u^{\lambda})}{u-u^{\lambda}}w=-c(x,\lambda)w,$$
where $c(x,\lambda)$ is bounded. Note that $\Sigma^{\lambda}$ consists of a part of $T^{\lambda}$ where $w=0$ and part of $\partial \Omega$ where $w=-u^{\lambda}<0$. In summary, we have 
\begin{eqnarray}
\label{eq17} &\Delta w +c(x,\lambda)w=0\,\,\text{in $\Sigma^{\lambda}$},\\
\nonumber &w\leq 0\,\,\text{and $w\not \equiv 0$ on $\partial \Sigma^{\lambda}$},
\end{eqnarray}
where $c(x,\lambda)$ is bounded.\\
\noindent Step 2. Beginning and ending.
For $a-\delta<\lambda<a,\, \Sigma^{\lambda}$ is narrow and so has small volume for $\delta$ small. Thus we may apply the ABP maximum to \eqref{eq17} and conclude that $w<0$ inside $\Sigma^{\lambda}$. Note also that by the Hopf boundary point lemma,
$$w_{x_1}|_{x_1=\lambda}=2u_{x_1}|x_1=\lambda<0.$$
Set 
$$\lambda_0=\inf \{0<\lambda<a:\,\,w<0 \,\,\text{in}\,\,\Sigma^{\lambda}\}.$$
If $\lambda_0=0$ we are done. Assume for contradiction that $\lambda_0>0$.\\
By continuity, $w\leq 0$ in $\Sigma^{\lambda_0}$ and moreover $w\not \equiv 0$ on $\Sigma^{\lambda_0}$.
Therefore the strong maximum principle implies
\be \label{eq18} w<0 \,\,\text{in $\Sigma^{\lambda_0}$}.\ee
\noindent Step 3: Deriving a contradiction.
\begin{claim}  $w^{\lambda_0-\e}<0$ in $\Sigma^{\lambda_0-\e}$ for sufficiently small $0<\e<\e_0$.
\end{claim}
Choose a simply connected closed set $K$ in $\Sigma^{\lambda_0}$ with smooth boundary (which is nearly all of $\Sigma^{\lambda_0}$ in measure) such that
$$|\Sigma^{\lambda_0}\setminus K|<\delta/2,$$
with $\delta$ as in Theorem \ref{ABP}.
By \eqref{eq18} there exists $\eta>0$ so that
\be w^{\lambda_0}\leq -\eta \,\,\text{for all $x\in K$},\ee
and so by continuity
\be \label{eq19} w^{\lambda_0-\e}\leq -\eta/2 \,\,\text{for all $x\in K$},\ee
for $\e$ sufficiently small.
Now we are in the situation that
$$w^{\lambda_0-\e} \leq 0  \,\,\text{on $\partial(\Sigma_{\lambda_0}\setminus K)$}. $$
However for $\e$ sufficiently small, we can make $|\Sigma_{\lambda_0-\e}\setminus K|<\delta$.\\
Once again applying the ABP maximum principle Theorem \ref{ABP} gives
$$w<0 \,\,\text{in $\Sigma_{\lambda_0-\e}\setminus K$}.$$
Combined with \eqref{eq19} we have
$$ w_{\lambda_0-\e}<0 \,\,\text{in $\Sigma_{\lambda_0-\e}$},$$
proving the Claim and giving a contradiction at last!

\begin{remark}

Louis' many contributions (with his collaborators) on the basic method of moving planes \cite{GNN1}, \cite{GNN2} and its variant ``the sliding method''   \cite{BN1} have become a standard tool in elliptic pde. He continued finding new twists and applications of moving planes to interesting pde problems in many papers with Berystycki, Caffarelli and Yanyan Li \cite{BN.mon1, BN.cyl, BCN1990, BN1991, BN1992, BCN1993, BCN.mon2, BCN1997, CYN2}.

His work has inspired thousands of papers and also some important variants and extensions to fully nonlinear elliptic pde \cite{CLi},
including a measure theoretic version \cite{CGS} \cite{CLi2} used to proved asymptotic symmetry  without growth assumptions at singularites (including infinity), and the method of ``moving spheres''  \cite{Li-Zhu, Li-Zhang, Jin-Li-Xu} which is particularly important for conformally invariant elliptic pde. There are also many, many other important contributions.

\end{remark}

\section{The Monge-Amp\`{e}re boundary value problem.}

The  Monge-Amp\`{e}re equation $\text{det}\, u_{ij}= \psi( x, u, \nabla u ) > 0$ is the prototypical fully nonlinear elliiptic pde. To see when it is elliptic, we linearize at a point $x_0$ , i.e.
$$L \phi = \frac{d}{dt} \text{det}\, [(u_{ij})(x_0) + t \phi_{ij}] |_{t=0} = A^{ ij} \phi_{ij},$$
where $A^{ij}$ is the cofactor matrix of $u_{ij}(x_0)$. For ellipticity we need $ ( A^{ij }) > 0$, that is, u convex.

The boundary value problem for the Monge-Amp\`{e}re operator is classically formulated as follows.
Let $\Omega\subset R^n$ be a smooth strictly convex domain and let $\phi, \, \psi>0$ be smooth.  Find a strictly convex solution 
$u\in C^{\infty}(\ol{\Omega})$ of the boundary value problem

\begin{eqnarray}
\label{eq4.10} \text{det}(u_{ij})&=&\psi(x,u,\nabla u)\,\,\text{in $\Omega$},\\
\label{eq4.20} u&=&\phi \,\,\text{on $\partial \Omega$}.
\end{eqnarray}

\subsection{Some history.} 

In 1971 Pogorelov showed that for the special case $\phi\equiv 0,\, \psi=\psi(x)>0$, there is a unique weak solution $u$ in the sense of Alexandrov and moreover $u\in C^{\infty}(\Omega)$. 
His proof uses his famous interior second derivative estimate and Calabi's ingenious interior estimates for third derivatives (in the metric $ds^2=u_{ij}dx_i dx_j$).\\

\begin{center}
\includegraphics[width=2.5cm, height=3cm]{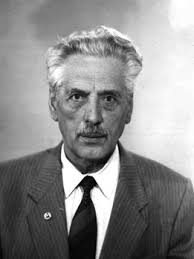}
 \hspace{.1in} 
 \includegraphics[width=2.4cm, height=3cm]{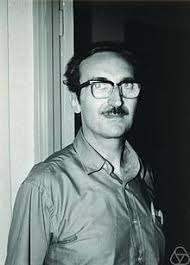}
\end{center}

 In 1974 Louis announced at the International Congress in Vancouver his joint work with Calabi providing a solution of \eqref{eq4.10}-\eqref{eq4.20} in complete generality. Unfortunately the argument contained a gap and fell through.

Louis, Luis Caffarelli \footnote{Louis recognized Luis' incredible brilliance and hired him away from U. Minnesota where he had spent the first ten years of his career. Luis spent only 1980--82 and 1994--1997 at Courant but became Louis' most frequent collaborator.} and I \cite{CNS1} \footnote{With Joe Kohn we also treated the complex Monge-Amp\`{e}re boundary value problem \cite{CKNS2}}  and independently Nikolai Krylov (see the expository article \cite{Krylov} and the references therein) proved the now classical existence theorem:

\begin{theorem} Suppose $\Omega$ is a strictly convex domain with $\Omega,\, \phi,\, \psi>0 $ smooth. Assume
in addition that for the boundary data $\phi$, there is a  strictly convex subsolution $\ul{u}$, i.e.
\vspace{-.1in}
\begin{eqnarray}
\det{(\ul{u}_{ij})}&\geq & \psi(x,\ul{u}, \nabla \ul{u})\,\,\text{in $\Omega$},\\
\nonumber  \ul{u}&=&\phi \,\, \,\text{on $\partial \Omega$}.
\end{eqnarray}

Then there exists a strictly convex solution $u\in C^{\infty}(\ol{\Omega})$ to \eqref{eq4.10}-\eqref{eq4.20}. If $\psi_u\geq 0$, the solution is unique.
\end{theorem}

\begin{center}
\includegraphics[width=2cm, height=2.5cm]{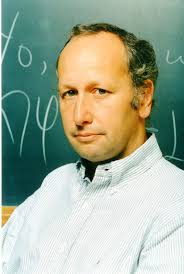}
 \hspace{.1in} 
 \includegraphics[width=2cm, height=2.5cm]{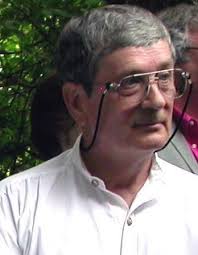}
\end{center}

\begin{remark}
There are two main difficulties in proving the existence of smooth strictly convex solutions via the continuity method in which one tries to prove apriori estimates for any admissible solution. The first one is to show that the second normal derivative
is apriori bounded  on the boundary (assuming all other second derivatives are apriori bounded), that is , $u_{nn}\leq C \,\text{ on $\partial \Omega$}$. The standard way to do this is from the equation $\det D^2u=\psi$. In a suitable frame one can write
$$A^{nn}u_{nn}=O(1)$$
where $A^{ij}$ is the cofactor matrix of $u_{ij}$ and then solve for $u_{nn}$. Thus one needs to prove strict convexity of the solution at the boundary.
 In the simplest case $\phi=0$, assuming strict convexity of the domain, this boils down to showing
$u_n\geq c_0>0$ with $c_0$ a controlled constant, for $e_n$ the outer unit normal.

The second major problem is that $C^2$ estimates do not suffice because the equation is fully nonlinear elliptic.
 One need to obtain global  $C^{2+\alpha}$ estimates. 
 This can be done strictly in the interior of $\Omega$ using the Evans-Krylov theorem  or Calabi's third derivative estimates. However the global estimates required a new idea.
 
\end{remark}

The following simple geometrical example suggests an improved result.
\begin{example} Let $\Gamma_1,\,\Gamma_0$ be strictly convex smooth closed codimension 2 hypersurfaces
in parallel planes, say $x_{n+1}=1,0$ respectively. Is there a hypersurface $S$ of constant Gauss curvature $K_0$  for $K_0$ sufficiently small? Intuitively the answer is clearly yes. 

Let's specialize further and suppose the parallel projection of $\Gamma_1$, call it $\gamma_1$ contains $\Gamma_0$. Then it is not difficult to see that if a solutions exists it is of the form  $S=\text{graph}(u)$  over the annulus $\Omega$ with outer boundary $\gamma_1$ and inner boundary $\Gamma_0$.
Thus $u$ satisfies \eqref{eq4.10} \eqref{eq4.20} with $\psi=K_0(1+|\nabla u|^2)^{\frac{n+2}2}$ and $\phi=1$ on $\gamma_1,\, \phi=0$ on $\Gamma_0$.
\end{example}

It was shown in \cite{HRS} that there is a unique smooth solution as expected. However since $\Omega$ is not convex, the classical existence theorem does not apply. This motivated Bo Guan and me \cite{GS, Guan1} to prove the subsolution existence theorem which holds in great generality.

\begin{theorem}Suppose $\Omega,\, \phi,\, \psi >0$ smooth and assume
that for the boundary data $\phi$, there is a  locally strictly convex subsolution $\ul{u}\in C^{\infty}(\ol{\Omega})$, i.e.
\vspace{-.1in}
\begin{eqnarray}
\det{(\ul{u}_{ij})}&\geq & \psi(x,\ul{u}, \nabla \ul{u})\,\,\text{in $\Omega$},\\
\nonumber  \ul{u}&=&\phi \,\, \,\text{on $\partial \Omega$}.
\end{eqnarray}

Then there exists a locally strictly convex solution $u\in C^{\infty}(\ol{\Omega})$ to \eqref{eq4.10}-\eqref{eq4.20}. If $\psi_u\geq 0$, the solution is unique.
Moreover any admissible solution satisfies the apriori estimate $\|u\|_{C^{2+\alpha}(\Omega)}\leq C$ for a controlled constant $C$.
\end{theorem}

\section{Implicitly defined fully nonlinear elliptic pde.}

The work of CNS on Monge-Ampere equations has a natural and  important extension to implicitly defined fully nonlinear pde.
Let $A=(a_{ij})$ be a symmetric nxn matrix (or more generally a natural tensor on a Riemannian manifold) and define
$$F(A)=f(\lambda_1,\ldots,\lambda_n),$$
where the $\lambda_i$ are the eigenvalues of $A$ and $f(\lambda)$ is a symmetric function (say smooth for simplicity). Then $F(A)$ will also be smooth.
When $A=(u_{ij}(x)),\,x\in \Omega$, and $f(\lambda)=\Pi \lambda_i,\, F(A)=\det{u_{ij}(x)}$, and we recover the Monge-Ampere operator which is elliptic when $\lambda$ lies in the positive cone $\Gamma^+_n=\{\lambda \in R^n: \lambda_i>0\}$.
\vspace{.1in}
What happens in general?

\subsection{Ellipticity and concavity.} 
Let $f$ is defined in a symmetric open convex cone $\Gamma$ in $\bfR^n$ with vertex at the origin with $\Gamma^+_n \subset \Gamma$. 
\begin{proposition} Assume that $f_i:=f_{\lambda_i}>0 \, \forall i$ and that $f$ is concave. Then $F^{ij}:=\frac{\partial F}{\partial a_{ij}}=f_i \delta_{ij}$ when $A$ is diagonal so the linearized operator $L=F^{ij}\nabla_i \nabla_j$ is elliptic. Moreover 
$F$ is a concave function of $A$.
\end{proposition}
The concavity  condition is very important for the regularity theory of fully nonlinear elliptic pde.

\subsection{G\r{a}rding's theory of hyperbolic polynomials.} 

Louis was an incredibly thorough mathematician (i.e he didn't miss much) with a vast knowledge of pde literature and knew G\r{a}rding's work on hyperbolic polynomials. As the name suggests, G\r{a}rding's beautiful theory was related to hyperbolic pde
but ultimately had many important algebraic consequences and is important in convex analysis (see the article of Reese Harvey and Blaine Lawson \cite{HL}).
\vspace{.1in}
\begin{center}
\includegraphics[width=2.5cm, height=3cm]{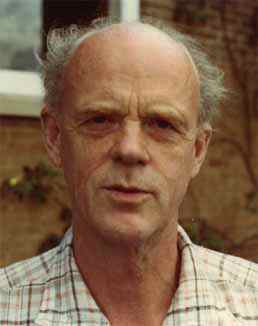}
\end{center}g

\begin{definition} 
A homogeneous polynomial $p(\lambda)$ of degree m in $R^n$ is called hyperbolic with respect to a direction $a\in R^n$ (notation $\text{hyp}\, a$) if for all $x\in R^n$, the polynomial $p(x+ t a)$ has exactly $k$ real roots. Thus
$$p(x+ta)=p(a)\Pi_{k=1}^m (t+\lambda_k(x)).$$
We may assume $p(a)>0$. It is easily checked that 
\vspace{-.1in}
$$q(\lambda):=\sum a_j \frac{\partial}{\partial \lambda_j}p(\lambda)$$

 is also $\text{hyp} \,a$.
 \end{definition}

\begin{example}  Since $\sigma_n(\lambda)=\Pi \lambda_i$ is $\text{hyp}\,a$ for $a=(1,\ldots,1)$, so are the elementary symmetric functions $\sigma_k(\lambda)$.
\end{example}
\vspace{.05in}

\vspace{.1in}
\subsection{The G\r{a}rding cone.} 
Let $\Gamma=\Gamma(P,a)$ denote the component in $R^n$ containing $a$ of the set where $p>0$. G\r{a}rding proved \cite{Garding} that $\Gamma$ is a convex cone with vertex at the origin and that $p$ is $\text{hyp}\,b$ for all $b\in \Gamma$.  Moreover
$\Gamma(p, a)\subset \Gamma(q,a)$.

 In particular the G\r{a}rding cones for the elementary symmetric functions $\sigma_k(\lambda)$ are nested starting from the positive cone for $\sigma_n(\lambda)$ and ending with the half-space $\sum \lambda_i>0$ for $\sigma_1(\lambda)$. Note also that for $k>1,\, \sigma_k=0$ on the positive $\lambda_i$ axes.

The main result proved by G\r{a}rding \cite{Garding} is an inequality which is equivalent to the statement that $p^{\frac1m}(\lambda)$ is concave in $\Gamma$. It follows easily that  $p_{\lambda_i}(\lambda)>0$  in $\Gamma$ for all $i=1,\ldots, n$.

\begin{remark} 

There are interesting and important examples of elliptic and concave $f(\lambda)$ that do not arise from hyperbolic polynomials
but are related to them. For example, 
the homogeneous degree 1 quotients
$$f(\lambda):= \left(\frac{\sigma_k(\lambda)}{\sigma_l(\lambda)}\right)^{\frac1{k-l}},\, k>l $$
 is elliptic and concave in the G\r{a}rding cone $\Gamma_k$ for $\sigma_k(\lambda)$.
\end{remark}

\subsection{The simplest version of the Dirichlet problem.} Let $u$ satisfy 
\begin{eqnarray*}
&F(D^2 u(x))\equiv f(\lambda(D^2 u(x)))=\psi(x)\,\,\text{ in $\Omega$},\\
&u=\varphi\,\,\text{on $\partial \Omega$},
\end{eqnarray*}
where $f(\lambda)$ is elliptic and concave in a symmetric open convex cone $\Gamma$  (with vertex at origin),
 containing the positive cone $\Gamma_n$ and $\limsup_{\lambda \goto \partial \Gamma}f(\lambda)\leq \inf_{\ol{\Omega}} \psi$.

 Our paper \cite{CNS3} assumes also the additional structural condition:\\
 for all $C>0$ and $K$ compact, there exist $R=R(C,K)$ such that
 $$f(\lambda_1, \ldots,\lambda_{n-1}, \lambda_n+R)\geq C,\,\, f(R\lambda)\geq C \,\,\text{for all $\lambda \in K$}.$$
 
\begin{theorem}  If there exists $R$ large so that $(\kappa_1,\ldots, \kappa_{n-1}, R) \in \Gamma$
at each $x\in \partial \Omega$, there exists a unique admissible solution $u\in C^{\infty}(\ol{\Omega})$.
Here $\kappa_1, \ldots,\kappa_{n-1}$ are the principal curvatures of $\partial \Omega$.\\
\end{theorem}
Our theorem was hard to prove but is far from optimal. It also excludes the nice example 
$$f(\lambda):= \left(\frac{\sigma_k(\lambda)}{\sigma_l(\lambda)}\right)^{\frac1{k-l}},\, k>l $$.

Over the years there has been significant improvements. For example in 2014, Bo Guan \cite{Guan2} proved the essentially optimal result:

\begin{theorem} Assume only ellipticity and concavity (i.e no additional structure conditions). If there exists an admissible subsolution $\ul{u}\in C^2(\ol{\Omega}),\,\ul{u}=\varphi \,\,\text{on $\partial \Omega$}$, then
there exists a unique admissible solution $u\in C^{\infty}(\ol{\Omega})$. If in addition $f(\lambda)$ satisfies
$$\sum \lambda_i f_{\lambda_i}\geq 0 \,\, \text{ on $\Gamma\cap\{\inf \psi \leq f \leq \sup \psi\}$},$$
we may take $\Omega$ to be a Riemannian manifold with smooth boundary.
\end{theorem}

\subsection{Final thoughts.}

The setting of implicitly defined fully nonlinear elliptic initiated in the paper \cite{CNS3} unleashed a tidal wave of research in fully nonlinear pde and geometric analysis (curvature flows, conformal geometry, complex geometry, \ldots) which is still very much ongoing.
Louis Nirenberg's scholarship, insight and experience played a large role in this development and it is certainly one of his most enduring legacies.

 Louis was also a  superb Ph.d advisor who produced 46 students\footnote{This is according to the math genealogy project tabulation.} and mentored  numerous young mathematicians from all over the world and this is also a great part of his  legacy. We can all learn from him about the benefits of generosity, openness and collaboration.\\


\bigskip

\begin{thebibliography}{199}

\bibitem{N.thesis} Nirenberg, Louis, The Weyl and Minkowski problems in differential geometry in the large. Comm. Pure Appl. Math. 6 (1953), 337–394.\\


\bibitem{N1953} Nirenberg, Louis, On nonlinear elliptic partial differential equations and Hölder continuity. Comm. Pure Appl. Math. 6 (1953), 103–156.\\




\bibitem{BersN1} Bers, Lipman and Nirenberg, Louis,  On a representation theorem for linear elliptic systems with discontinuous coefficients and its applications. Convegno Internazionale sulle Equazioni Lineari alle Derivate Parziali, Trieste, 1954, pp. 111–140. Edizioni Cremonese, Roma, 1955. \\

\bibitem{BersN2}  Bers, Lipman and Nirenberg, Louis,   On linear and non-linear elliptic boundary value problems in the plane. Convegno Internazionale sulle Equazioni Lineari alle Derivate Parziali, Trieste, 1954, pp. 141–167. Edizioni Cremonese, Roma, 1955.  \\




\bibitem{DoN} Douglis, Avron and Nirenberg, Louis, Interior estimates for elliptic systems of partial differential equations. Comm. Pure Appl. Math. 8 (1955), 503–538.\\

\bibitem{N1955} Nirenberg, Louis, Remarks on strongly elliptic partial differential equations. Comm. Pure Appl. Math. 8 (1955), 649–675. \\


\bibitem{MoN} Morrey, Charles B. and Nirenberg, Louis,   On the analyticity of the solutions of linear elliptic systems of partial differential equations. Comm. Pure Appl. Math. 10 (1957), 271–290. \\



\bibitem{NewN} Newlander, August and  Nirenberg, Louis,   Complex analytic coordinates in almost complex manifolds. Ann. of Math. (2) 65 (1957), 391–404. \\

\bibitem{KoNS} Kodaira, Kunihiko, Nirenberg, Louis and Spencer, Donald C.,   On the existence of deformations of complex analytic structures. Ann. of Math. (2) 68 (1958), 450–459. \\



\bibitem{HaN}Hartman, Philip and Nirenberg, Louis,    On spherical image maps whose Jacobians do not change sign. Amer. J. Math. 81 (1959), 901–920. \\

 \bibitem{ADN1} Agmon, Shmuel,  Douglis, Avron and Nirenberg, Louis,    Estimates near the boundary for solutions of elliptic partial differential equations satisfying general boundary conditions. I. Comm. Pure Appl. Math. 12 (1959), 623–727. \\
 
 
\bibitem{ADN2}  Agmon, Shmuel, Douglis, Avron and Nirenberg, Louis,  Estimates near the boundary for solutions of elliptic partial differential equations satisfying general boundary conditions. II. Comm. Pure Appl. Math. 17 (1964), 35–92. \\
 
\bibitem{N1959} Nirenberg, Louis,  On elliptic partial differential equations. Ann. Scuola Norm. Sup. Pisa Cl. Sci. 13 (1959), 115--162. \\







\bibitem{N1961} Nirenberg, Louis,  Comments on elliptic partial differential equations. 1961 Proc. Sympos. Pure Math., Vol. IV pp. 101–108 American Mathematical Society, Providence, R.I. \\

\bibitem{JoN} John, Fritz and Nirenberg, Louis,   On functions of bounded mean oscillation. Comm. Pure Appl. Math. 14 (1961), 415–426.\\ 

\bibitem{TN} Nirenberg, Louis and Tr\`{e}ves, Francois,  Solvability of a first order linear partial differential equation. Comm. Pure Appl. Math. 16 (1963), 331--351. \\


\bibitem{KoN1} Kohn, Joseph J. and Nirenberg, Louis,  Non-coercive boundary value problems. Comm. Pure Appl. Math. 18 (1965), 443–492.\\
 
\bibitem{KoN2} Kohn, Joseph J. and Nirenberg, Louis,  An algebra of pseudo-differential operators. Comm. Pure Appl. Math. 18 (1965), 269–305.\\

\bibitem{KoN3} Kohn, Joseph J. and Nirenberg, Louis,  Degenerate elliptic-parabolic equations of second order. Comm. Pure Appl. Math. 20 (1967), 797–872. \\



\bibitem{CLN} Chern, S. S., Levine, Harold I. and Nirenberg, Louis,  Intrinsic norms on a complex manifold. 1969 Global Analysis (Papers in Honor of K. Kodaira) 119–139 Univ. Tokyo Press, Tokyo \\



\bibitem{TN1} Nirenberg, Louis and Trèves, Francois, On local solvability of linear partial differential equations. I. Necessary conditions. Comm. Pure Appl. Math. 23 (1970), 1–38. \\

\bibitem{TN2} Nirenberg, Louis and Trèves, Francois,  On local solvability of linear partial differential equations. II. Sufficient conditions. Comm. Pure Appl. Math. 23 (1970), 459–509. \\






 



\bibitem{LN} Loewner, Charles and  Nirenberg, Louis, Partial differential equations invariant under conformal or projective transformations. Contributions to analysis (a collection of papers dedicated to Lipman Bers), pp. 245–272. Academic Press, New York, 1974. \\









\bibitem{KN} Kinderlehrer, David and Nirenberg, Louis, Regularity in free boundary problems. Ann. Scuola Norm. Sup. Pisa Cl. Sci. (4) 4 (1977), no. 2, 373–391. \\

\bibitem{KNS1} Kinderlehrer, David, Nirenberg, Louis and Spruck, Joel,  Regularity in elliptic free boundary problems. J. Analyse Math. 34 (1978), 86–119 (1979). \\


\bibitem{KNS2}Kinderlehrer, David.; Nirenberg, Louis and Spruck, Joel,  Regularity in elliptic free boundary problems. II. Equations of higher order. Ann. Scuola Norm. Sup. Pisa Cl. Sci. (4) 6 (1979), no. 4, 637–683. \\


\bibitem{GNN1}Gidas, Basilis, Ni, Wei-Ming and Nirenberg, Louis,  Symmetry and related properties via the maximum principle. Comm. Math. Phys. 68 (1979), no. 3, 209–243. \\


\bibitem{GNN2} Gidas, Basilis, Ni, Wei-Ming and Nirenberg, Louis, Symmetry of positive solutions of nonlinear elliptic equations in $\bfR^n$. Mathematical analysis and applications, Part A, pp. 369–402, Adv. in Math. Suppl. Stud., 7a, Academic Press, New York-London, 1981. \\


\bibitem{NWY} Nirenberg, Louis, Webster, Sidney and Yang, Paul,  Local boundary regularity of holomorphic mappings. Comm. Pure Appl. Math. 33 (1980), no. 3, 305–338. \\


\bibitem{CKoN} Caffarelli, Luis, Kohn, Robert and Nirenberg, Louis, Partial regularity of suitable weak solutions of the Navier-Stokes equations. Comm. Pure Appl. Math. 35 (1982), no. 6, 771–831. \\



\bibitem{BN1} Brezis, Haim and Nirenberg, Louis,  Positive solutions of nonlinear elliptic equations involving critical Sobolev exponents. Comm. Pure Appl. Math. 36 (1983), no. 4, 437–477. \\



\bibitem{CNS1} Caffarelli, Luis, Nirenberg, Louis and Spruck, Joel,  The Dirichlet problem for nonlinear second-order elliptic equations. I. Monge-Ampère equation. Comm. Pure Appl. Math. 37 (1984), no. 3, 369–402. \\


\bibitem{CKNS2}Caffarelli, Luis, Kohn, Joseph J., Nirenberg, Louis and  Spruck, Joel,  The Dirichlet problem for nonlinear second-order elliptic equations. II. Complex Monge-Ampère, and uniformly elliptic, equations. Comm. Pure Appl. Math. 38 (1985), no. 2, 209–252. \\


\bibitem{CNS3} Caffarelli, Luis, Nirenberg, Louis and Spruck, Joel,  The Dirichlet problem for nonlinear second-order elliptic equations. III. Functions of the eigenvalues of the Hessian. Acta Math. 155 (1985), no. 3-4, 261–301.\\

\bibitem{CNS4} Caffarelli, Luis, Nirenberg, Louis and Spruck, Joel,  Nonlinear second order elliptic equations. IV. Starshaped compact Weingarten hypersurfaces. Current topics in partial differential equations, 1–26, Kinokuniya, Tokyo, 1986.\\

\bibitem{CNS.degenerate} Caffarelli, Luis, Nirenberg, Louis and Spruck, Joel,  The Dirichlet problem for the degenerate Monge-Ampère equation. Rev. Mat. Iberoamericana 2 (1986), no. 1-2, 19–27. \\


\bibitem{CNS.bernstein} Caffarelli, Luis, Nirenberg, Louis and Spruck, Joel,   On a form of Bernstein's theorem. Analyse mathématique et applications, 55–66, Gauthier-Villars, Montrouge, 1988. \\


\bibitem{CNS5} Caffarelli, Luis, Nirenberg, Louis and Spruck, Joel,  Nonlinear second-order elliptic equations. V. The Dirichlet problem for Weingarten hypersurfaces. Comm. Pure Appl. Math. 41 (1988), no. 1, 47–70.\\




 \bibitem{BN.mon1} Berestycki, Henri and Nirenberg, Louis,   Monotonicity, symmetry and antisymmetry of solutions of semilinear elliptic equations. J. Geom. Phys. 5 (1988), no. 2, 237–275. \\

\bibitem{BCN1990} Berestycki, Henri, Caffarelli, Luis and Nirenberg, Louis, Uniform estimates for regularization of free boundary problems. Analysis and partial differential equations, 567–619, Lecture Notes in Pure and Appl. Math., 122, Dekker, New York, 1990. \\


\bibitem{BN.cyl} Berestycki, Henri and Nirenberg, Louis, Some qualitative properties of solutions of semilinear elliptic equations in cylindrical domains. Analysis, et cetera, 115–164, Academic Press, Boston, MA, 1990. \\

\bibitem{BN.sliding} Berestycki, Henri and Nirenberg, Louis, On the method of moving planes and the sliding method. Bol. Soc. Brasil. Mat. (N.S.) 22 (1991), no. 1, 1–37. \\

\bibitem{BN1991} Berestycki, Henri and  Nirenberg, Louis, Asymptotic behaviour via the Harnack inequality. Nonlinear analysis, 135–144, Sc. Norm. Super. di Pisa Quaderni, Scuola Norm. Sup., Pisa, 1991. \\

\bibitem{BN1992} Berestycki, Henri and Nirenberg, Louis, Traveling fronts in cylinders. Ann. Inst. H. Poincaré Anal. Non Linéaire 9 (1992), no. 5, 497–572. \\


\bibitem{BCN1993} Berestycki, Henri,  Caffarelli, Louis and Nirenberg, Louis, Symmetry for elliptic equations in a half space. Boundary value problems for partial differential equations and applications, 27–42, RMA Res. Notes Appl. Math., 29, Masson, Paris, 1993. \\


\bibitem{BCN.ineq} Berestycki, Henri, Caffarelli, Luis and Nirenberg, Louis, Inequalities for second-order elliptic equations with applications to unbounded domains. I. A celebration of John F. Nash, Jr. Duke Math. J. 81 (1996), no. 2, 467–494. \\

\bibitem{BCN.mon2} Berestycki, Henri,  Caffarelli, Luis and  Nirenberg, Louis, Monotonicity for elliptic equations in unbounded Lipschitz domains. Comm. Pure Appl. Math. 50 (1997), no. 11, 1089–1111. \\


\bibitem{BCN1997} Berestycki, Henri; Caffarelli, Luis and Nirenberg, Louis, Further qualitative properties for elliptic equations in unbounded domains. Dedicated to Ennio De Giorgi. Ann. Scuola Norm. Sup. Pisa Cl. Sci. (4) 25 (1997), no. 1-2, 69–94 (1998). \\

\bibitem{BN1997} Brezis, Haim and Nirenberg, Louis,  Removable singularities for nonlinear elliptic equations. Topol. Methods Nonlinear Anal. 9 (1997), no. 2, 201–219.\\

\bibitem{YYN1998} Li, Yanyan and Nirenberg, Louis, The Dirichlet problem for singularly perturbed elliptic equations. Comm. Pure Appl. Math. 51 (1998), no. 11-12, 1445–1490.\\



\bibitem{LiN.distance} Li, Yanyan and Nirenberg, Louis, The distance function to the boundary, Finsler geometry, and the singular set of viscosity solutions of some Hamilton-Jacobi equations. Comm. Pure Appl. Math. 58 (2005), no. 1, 85–146.\\



\bibitem{LiN.hopf1} Li, Yanyan and Nirenberg, Louis, A geometric problem and the Hopf lemma. I. J. Eur. Math. Soc. (JEMS) 8 (2006), no. 2, 317–339. \\

\bibitem{LiN.hopf2} Li, Yanyan and Nirenberg, Louis,  A geometric problem and the Hopf lemma. II. Chinese Ann. Math. Ser. B 27 (2006), no. 2, 193–218.\\



\bibitem{CYN1} Caffarelli, Luis, Li, Yanyan and Nirenberg, Louis, Some remarks on singular solutions of nonlinear elliptic equations. I. J. Fixed Point Theory Appl. 5 (2009), no. 2, 353–395.\\

 \bibitem{CYN2}Caffarelli, Luis, Li, Yanyan and Nirenberg, Louis, Some remarks on singular solutions of nonlinear elliptic equations. II. Symmetry and monotonicity via moving planes. Advances in geometric analysis, 97–105, Adv. Lect. Math. (ALM), 21, Int. Press, Somerville, MA, 2012. \\

\bibitem{CYN3}Caffarelli, Luis,  Li, Yanyan and Nirenberg, Louis, Some remarks on singular solutions of nonlinear elliptic equations III: viscosity solutions including parabolic operators. Comm. Pure Appl. Math. 66 (2013), no. 1, 109--143. \\

\bibitem{Garding} G\r{a}rding, Lars, An inequality for Hyperbolic polynomials, J. Math and Mechanics 8 (1959), 957--965.\\


\bibitem{Taylor} Taylor, Jean E., The structure of singularities in soap-bubble-like and soap-film-like minimal surfaces.
Ann. of Math. 103 (1976), 489--539.\\

\bibitem{Jellett1}  Jellett, John Hewitt,  On the properties of inextensible surfaces, Trans. Irish. Ac. Dublin 22 (1849):
343--377.\\


\bibitem{Jellett2} Jellett, John Hewitt, Sur la surface dont la courbure moyenne est constante, J. Math.
Pures et Appl.  8 (1853) 163-167.\\

\bibitem{ball} Spruck, Joel, On the radius of the smallest ball containing a compact manifold of positive curvature.
J. Differential Geometry 8 (1973), 257–258.\\

\bibitem{Nitsche} Nitsche, Johannes C.C., The higher regularity of liquid edges in aggregates of minimal surfaces.
Nachr. Akad. Wiss. Göttingen Math.-Phys. Kl. II (1978), 31--51.\\

\bibitem{GS} Gidas, Basilis and Spruck, Joel, Global and local behavior of positive solutions of nonlinear elliptic equations. Comm. Pure Appl. Math. 34 (1981), 525--598. \\

\bibitem{Hsiang} Hsiang, Wu-Yi,  Generalized rotational hypersurfaces of constant mean curvature in the Euclidean spaces. I, J. Differential Geometry 17 (1982), 337–356.\\

\bibitem{Wente} Wente, Henry C., 
Counterexample to a conjecture of H. Hopf, Pacific J. Math. 121 (1986), 193–243.\\

\bibitem{CGS}  Caffarelli, Luis, Gidas, Basilis and  Spruck, Joel,  Asymptotic symmetry and local behavior of semilinear elliptic equations with critical Sobolev growth. Comm. Pure Appl. Math. 42 (1989), 271--297.\\

\bibitem{Krylov} Krylov, Nicolai V.,
Fully nonlinear second order elliptic equations: recent development. 
Ann. Scuola Norm. Sup. Pisa Cl. Sci. 25 (1997), 569--595.\\

\bibitem{CLi}  Li, Congming, Monotonicity and symmetry of solutions of fully nonlinear elliptic equations on bounded domains. Comm. Partial Differential Equations 16 (1991), 491–526.\\

\bibitem{CLi2} Li, Congming, Local asymptotic symmetry of singular solutions to nonlinear elliptic equations.
Invent. Math. 123 (1996), no. 2, 221–231.\\

\bibitem{Li-Zhu} Li, Yanyan and Zhu, Meijun, Uniqueness theorems through the method of moving spheres. Duke Math. J. 80 (1995), 383--417. \\

\bibitem{Li-Zhang} Li, Yanyan and Zhang, Lei, Liouville-type theorems and Harnack-type inequalities for semilinear elliptic equations. J. Anal. Math. 90 (2003), 27--87.\\

\bibitem{Chen-CLi} Chen, Wenxiong and Li, Congming,  Moving planes, moving spheres, and a priori estimates. J. Differential Equations 195 (2003), 1–13. \\

\bibitem{Jin-Li-Xu} Jin, Qinian, Li, Yanyan and Xu, Haoyuan, Symmetry and asymmetry: the method of moving spheres. Adv. Differential Equations 13 (2008),  601–640. \\

\bibitem{HRS}  Hoffman, David, Rosenberg, Harold and Spruck, Joel,  Boundary value problems for surfaces of constant Gauss curvature,  Comm. Pure Appl. Math. 45 (1992), 1051–1062.\\

\bibitem{GS} Guan, Bo and Spruck, Joel,  Boundary-value problems on $\bfS^n$ for surfaces of constant Gauss curvature. Ann. of Math. 138 (1993), 601--624.\\

\bibitem{Guan1} Guan, Bo,  The Dirichlet problem for Monge-Amp\`{e}re equations in non-convex domains and spacelike hypersurfaces of constant Gauss curvature, Trans. Amer. Math. Soc. 350 (1998), 4955--4971.\\

\bibitem{GT} Gilbarg, David and Trudinger, Neil S,  Elliptic partial differential equations of second order,  Reprint of the 1998 edition. Classics in Mathematics. Springer-Verlag, Berlin, 2001.\\

\bibitem{Cabre} Cabr\'{e}, Xavier, Elliptic PDEs in Probability and Geometry: Symmetry and regularity of solutions, Discrete Contin. Dyn. Syst. 20 (2008), 425--457.\\

\bibitem{HL} Harvey, F. Reese and Lawson, H. Blaine, Jr.,  G\r{a}rding's theory of hyperbolic polynomials, Comm. Pure Appl. Math. 66 (2013), 1102--128. \\

\bibitem{Guan2} Guan, Bo, The Dirichlet problem for fully nonlinear elliptic equations on a Riemannian manifold, Duke Math. J. 163 (2014), 1491--1524.\\

\bibitem{Gutierrez}  Guti\'{e}rrez, Cristian E., The Monge-Amp\`{e}re equation. Second edition,  Progress in Nonlinear Differential Equations and their Applications, 89. Birkh\"{a}user/Springer, 2016. \\

\bibitem{Figalli} Figalli, Alessio, 
The Monge-Ampère equation and its applications.
Zurich Lectures in Advanced Mathematics. European Mathematical Society, Z\"{u}rich, 2017.\\

\bibitem{1996.medal} Caffarelli, Luis and Kohn, Joseph J., Louis Nirenberg Receives National Medal of Science, Notices AMS 43 (1996), 1111--1116.\\



\bibitem{YYLi} Li, Yanyan, The work of Louis Nirenberg, Proceedings of the International Congress of Mathematicians 2010 (ICM 2010), 126--137.\\


\bibitem{Don} Donaldson, Simon, On the work of Louis Nirenberg, Notices AMS  58 (2011), 469--472.\\


\bibitem{Riviere} Rivi\`{e}re, Tristan, Exploring the Unknown: The Work of Louis Nirenberg on Partial Differential Equations, Notices of the AMS 120--125 and full version: arxiv.org/abs/1505.04930.\\

\bibitem{AMS} Sormani, Christina, Recent applications of Nirenberg's classical ideas (contributions by Xavier Cabr\'{e}, Sun-Yung Alice Chang, Gregory Seregin, Eric Carlen and Alessio Figalli, Mu-Tao Wang and Shing-Tung Yau, Notices of the AMS 63 (2016), 
126--134.\\

\bibitem{Kohn.Abel} Kohn, Robert V.,  A few of Louis Nirenberg’s many contributions to the theory of partial differential equations,
chapter in part III, The Abel Prize 2013-2017, Springer 2019.\\




\bibitem{Vasquez} Vasquez, Juan Luis, Remembering Louis Nirenberg and his mathematics, extended English version of 
the article in Spanish,  La Gaceta de la Real Sociedad Matem\'{a}tica Espa\~{n}ola, vol. 23, June 2020,  243--261.\\


\end{thebibliography}
\end{document}